\documentclass[12pt,reqno,a4paper]{amsart}
\usepackage{cite}
\usepackage{pst-all}
\usepackage{amssymb}

\title[On the isoperimetric problem]{Some remarks on the isoperimetric
  problem for the higher eigenvalues of the Robin and Wentzell
  Laplacians}

\date{}

\author{J. B. Kennedy}
\dedicatory{\upshape
School of Mathematics and Statistics,
University of Sydney, NSW 2006, Australia\\[.5em]
\texttt{J.Kennedy@maths.usyd.edu.au}}

\newtheorem{theorem}{Theorem}[section]
\newtheorem{lemma}[theorem]{Lemma}
\newtheorem{proposition}[theorem]{Proposition}

\theoremstyle{definition}

\newtheorem{remark}[theorem]{Remark}

\numberwithin{equation}{section}

\DeclareMathOperator{\dist}{dist}

\DeclareMathOperator{\divergence}{div}

\newcommand{\R}{\mathbb{R}}
\newcommand{\N}{\mathbb{N}}

\DeclareMathAccent{\ocirc}{\mathalpha}{operators}{"17}

\begin{document}

\begin{abstract}
  We consider the problem of minimising the $k$th eigenvalue, $k \geq
  2$, of the ($p$-)Laplacian with Robin boundary conditions with
  respect to all domains in $\R^N$ of given volume $M$.  When $k=2$,
  we prove that the second eigenvalue of the $p$-Laplacian is
  minimised by the domain consisting of the disjoint union of two
  balls of equal volume, and that this is the unique domain with this
  property.  For $p=2$ and $k \geq 3$, we prove that in many cases a
  minimiser cannot be independent of the value of the constant
  $\alpha$ in the boundary condition, or equivalently of the volume
  $M$. We obtain similar results for the Laplacian with generalised
  Wentzell boundary conditions $\Delta u + \beta \frac{\partial
    u}{\partial \nu} + \gamma u = 0$.
\end{abstract}



\maketitle

\section{Introduction}
\label{sec:intro}

We are interested in the eigenvalue problem
\begin{equation}
  \label{eq:robinproblem}
  \begin{aligned}
    -\divergence(|\nabla u|^{p-2}\nabla u) &= \lambda |u|^{p-2}u
    \qquad &\text{in $\Omega$},\\
    |\nabla u|^{p-2}\frac{\partial u}{\partial \nu} + \alpha
    |u|^{p-2}u &=0 &\text{on $\partial \Omega$},
  \end{aligned}
\end{equation}
where $\Omega \subset \R^N$ is a bounded, Lipschitz domain,
$1<p<\infty$, $\alpha > 0$, and $\nu$ is the outward pointing unit
normal to $\Omega$. Here $\Delta_p u := \divergence (|\nabla
u|^{p-2}\nabla u)$ is the $p$-Laplacian of $u$ and the boundary
conditions in \eqref{eq:robinproblem} are of Robin type.

It is known that if $\Omega$ is connected, then analogous to the case
of Dirichlet boundary conditions there is an isolated simple first
eigenvalue $\lambda_1 = \lambda_1(\Omega, \alpha) > 0$ such that only
eigenfunctions associated with $\lambda_1$ do not change sign.
Moreover, there is a well-defined second eigenvalue $\lambda_2 >
\lambda_1$ at the base of the rest of the spectrum obtainable by the
L-S principle (see \cite[Section~5.5]{le:06:pl}). If $p=2$, then we
recover the usual sequence of eigenvalues $0 < \lambda_1 < \lambda_2
\leq \lambda_3 \leq \ldots \to \infty$ exhausting the spectrum (see
for example \cite{daners:00:rbv}). For not necessarily connected
domains $\Omega$, we wish to study minimisation problems of the form
\begin{equation}
  \label{eq:minproblem}
  \min \, \{\lambda_k(\Omega, \alpha): \Omega \subset \R^N \text{\ is 
    bounded, Lipschitz,\ } |\Omega| = M\}
\end{equation}
where $M>0$ and $\alpha > 0$ are fixed, $k \geq 2$ if $p=2$ and $k=2$
otherwise, and $|\,.\,|$ is $N$-dimensional Lebesgue measure. Note
that we list repeated eigenvalues according to their multiplicities.
Such problems are often called isoperimetric problems as they depend
on the geometry of the underlying domain.

When $k=1$ the Faber-Krahn inequality asserts that the unique solution
to \eqref{eq:minproblem} is a ball $B$ with $|B|=M$ (see
\cite{dai:08:rpl,bucur:09:aa}). When $k=2$ and $p=2$ it was proved in
\cite{kennedy:09:lr2} that a solution to \eqref{eq:minproblem}, which
we shall call $D_2$, is disjoint union of two equal balls of volume
$M/2$.

For $k=2$, it was proved in \cite{kennedy:08:wfk} that the domain
which we shall call $D_2$, consisting of the disjoint union of two
equal balls of volume $M/2$, is a solution to \eqref{eq:minproblem}
when $p=2$. Our first goal here is to generalise this result to all
$1<p<\infty$, and at the same time prove uniqueness of this minimiser
(that is, sharpness of the associated inequality). This is done in
Section~\ref{sec:second} (see Theorem~\ref{th:rp2}).

We consider the problem \eqref{eq:minproblem} for $k \geq 3$ in
Section~\ref{sec:rhigher}. Here we restrict our attention to the case
$p=2$ because the spectrum of the $p$-Laplacian is not well understood
otherwise. In particular, it is not known if the L-S sequence exhausts
the spectrum, although we expect our observations to generalise easily
if this is the case. We prove that for many values of $N$ and $k$
there cannot be a solution \eqref{eq:minproblem} independent of
$\alpha>0$ in \eqref{eq:robinproblem}, or equivalently, of the volume
$M>0$. (See Theorem~\ref{th:robink}.) Note that actually proving the
existence of a solution to \eqref{eq:minproblem} in general is an
extremely difficult problem -- this has not even yet been proved in
the easier Dirichlet case (see \cite{bucur:00:3de,henrot:03:min}), and
the Robin problem lacks many of the properties of the Dirichlet
problem (see Remark~\ref{rem:higherexamples}).

In Section~\ref{sec:whigher}, we consider the Laplacian with
generalised Wentzell boundary conditions
\begin{equation}
  \label{eq:wentzellproblem}
  \begin{aligned}
    -\Delta u&=\Lambda u \quad&\text{in $\Omega$},\\
    \Delta u +\beta\frac{\partial u}{\partial\nu}+\gamma u&=0
    \quad &\text{on $\partial \Omega$},
  \end{aligned}
\end{equation}
where $\beta, \gamma > 0$. Here too there exists a sequence of
eigenvalues $0 < \Lambda_1(\Omega) \leq \Lambda_2(\Omega) \leq \ldots$
exhausting the spectrum. Moreover, the first eigenvalue $\Lambda_1$
satisfies the (sharp) Faber-Krahn inequality $\Lambda_1 (\Omega) \geq
\Lambda_1 (B)$ for all bounded, Lipschitz $\Omega \subset \R^N$ as the
solution for $k=1$ to the analogue of \eqref{eq:minproblem} (see
\cite{kennedy:08:wfk}). This is a similar problem to
\eqref{eq:robinproblem}, and we prove analogues of our results for the
Robin problem in this case (see Theorem~\ref{th:wentzellk}). Here we
only consider the case $p=2$; it appears no work has yet been done on
developing a theory of the $p$-Laplacian with boundary conditions
$\Delta_p u + \beta |\nabla u|^{p-2} \frac{\partial u}{\partial \nu} +
\gamma |u|^{p-2}u =0$ on $\partial \Omega$.

Before we proceed, we have a few general remarks.

\begin{remark}
  \label{rem:general}
  (i) We will only consider bounded, Lipschitz domains of fixed volume
  $M>0$ unless otherwise specified, since this is in some sense the
  ``natural'' setting for problems such as \eqref{eq:robinproblem} and
  \eqref{eq:wentzellproblem}, although a solution to
  \eqref{eq:minproblem} could be unbounded or non-Lipschitz.

  (ii) We allow our domains to be disconnected, which is necessary for
  considering problems such as \eqref{eq:minproblem}. We will assume
  throughout that $\Omega$ consists of countably many bounded
  connected components (c.c.s for short), each having Lipschitz
  boundary, and that there exists $\delta>0$ such that the distance
  between any two c.c.s is at least $\delta$. Such domains are
  slightly more general than ``bounded, Lipschitz''. In such a case
  the eigenvalues of $\Omega$ (for any operator or boundary condition)
  can be found by collecting and reordering the eigenvalues of the
  c.c.s.

  (iii) For such domains $U,V$, in a slight abuse of notation we will
  say $U=V$ iff their c.c.s are in bijective correspondence and for
  each pair $\widetilde U, \widetilde V$ of c.c.s, there exists a
  rigid transformation $\tau$ such that $\tau(\widetilde U) =
  \widetilde V$. (Thus their spectra will coincide.)

  (iv) We will always use $\lambda = \lambda_k (\Omega, \alpha)$ to
  stand for an eigenvalue of \eqref{eq:robinproblem}, $\Lambda =
  \Lambda_k (\Omega, \beta, \gamma)$ for \eqref{eq:wentzellproblem},
  although we will drop one or more arguments if there is no danger of
  confusion, and we will denote by $\mu_k = \mu_k (\Omega)$ the $k$th
  eigenvalue of the Dirichlet $p$-Laplacian on $\Omega$. We collect
  some elementary properties of these eigenvalues in the appendix.
\end{remark}

\section{The second eigenvalue of the Robin $p$-Laplacian}
\label{sec:second}

Choose $1<p<\infty$, $\alpha>0$ and $M>0$, which will all be fixed for
this section. Let $\lambda_2(\Omega)$ be the second eigenvalue of
\eqref{eq:robinproblem} on $\Omega$, and let $D_2$ be the disjoint
union of two balls of volume $M/2$ each.

\begin{theorem}
  \label{th:rp2}
  Suppose $\Omega \subset \R^N$ is a domain of volume $M$ satisfying
  the assumptions of Remark~\ref{rem:general}(ii). Then
  $\lambda_2(\Omega) \geq \lambda_2(D_2)$ with equality if and only if
  $\Omega = D_2$ in the sense of Remark~\ref{rem:general}(iii).
\end{theorem}

To prove Theorem~\ref{th:rp2} we cannot directly apply the method used
in the Dirichlet case (see for example \cite[Section~4]{henrot:03:min}
and also \cite[Section~2]{kennedy:09:lr2} for when $p=2$; the
arguments are the same when $p \neq 2$) since the nodal domains may
not be smooth enough to apply the Faber-Krahn inequality, which is
only known for Lipschitz domains (see \cite{bucur:09:aa}). The proof
we give is a refinement of that in \cite{kennedy:09:lr2}, which for
$p=2$ constructs an appropriate sequence of approximations to the
nodal domain. A significant additional argument is needed to prove
uniqueness of the minimiser.

\begin{remark}
  \label{rem:rp2}
  When $p=2$, Theorem~\ref{th:rp2} combined with
  \cite[Example~2.2]{kennedy:09:lr2} shows that there is no minimiser
  of $\lambda_2$ amongst all \emph{connected} domains of given volume,
  since we can find a sequence of connected $\Omega_n$ with $\lambda_2
  (\Omega_n) \to \lambda_2 (D_2)$. A similar construct should work
  when $p \neq 2$, but we do not know of domain approximation results
  akin to those in \cite{dancer:97:dpr} for this case.
\end{remark}

Before we proceed with the proof of Theorem~\ref{th:rp2}, we recall
some properties of the eigenvalues and eigenfunctions of the problem
\eqref{eq:robinproblem}. Here for simplicity we will assume $\Omega$
is connected. We understand an eigenvalue $\lambda \in \R$ of
\eqref{eq:robinproblem} with eigenfunction $\psi \in W^{1,p} (\Omega)$
in the weak sense, as a solution of
\begin{equation}
  \label{eq:eigenmeaning}
  \int_\Omega |\nabla \psi|^{p-2}\nabla \psi \cdot \nabla \varphi \,dx
  + \int_{\partial \Omega} \! \alpha |\psi|^{p-2} \psi \varphi \,d\sigma
  = \lambda \int_\Omega |\psi|^{p-2} \psi \varphi \,dx
\end{equation}
for all $\varphi \in W^{1,p}(\Omega)$.

\begin{proposition}
  \label{prop:welldef}
  Suppose $\Omega \subset \R^N$ is a bounded, connected Lipschitz
  domain. Then
  \begin{itemize}
  \item[(i)] there exists a sequence of eigenvalues $(\lambda_n)_{n
      \in \N}$ of \eqref{eq:robinproblem}, obtainable by the
    Ljusternik-Schnirelman (L-S) principle, of the form
    $0<\lambda_1<\lambda_2 \leq \ldots$;
  \item[(ii)] the second L-S eigenvalue satisfies
    \begin{displaymath}
      \lambda_2 = \inf\{ \lambda > \lambda_1: \lambda \text{ is an 
        eigenvalue of \eqref{eq:robinproblem}} \};
    \end{displaymath}
  \item[(iii)] the first eigenvalue $\lambda_1>0$ is simple and every
    eigenfunction $\psi$ associated with $\lambda_1$ satisfies
    $\psi>0$ or $\psi<0$ in $\Omega$;
  \item[(iv)] only eigenfunctions associated with $\lambda_1$ do not
    change sign in $\Omega$;
  \item[(v)] every eigenfunction $\psi$ of \eqref{eq:robinproblem}
    lies in $W^{1,p}(\Omega) \cap C^{1,\eta}(\Omega) \cap C(\overline
    \Omega)$ for some $0<\eta<1$.
  \end{itemize}
\end{proposition}

\begin{proof}
  Parts (i)-(iv) are essentially contained in \cite{le:06:pl}.
  Although $C^1$ regularity of $\Omega$ is assumed there in order to
  derive (i) and $C^{1,\theta}$, $0<\theta<1$, is assumed for
  (ii)-(iv), a careful analysis of the proofs shows that only
  Lipschitz continuity of $\partial \Omega$ is needed, since all
  background results, including those in the appendices, are valid for
  Lipschitz domains. (The extra regularity of $\partial \Omega$ is
  needed only to prove extra boundary regularity of the
  eigenfunctions.) For (v), first note that by
  \cite[Theorem~2.7]{daners:09:ql}, every eigenfunction $\psi \in
  L^\infty (\Omega)$ (see also Section~4 there). But now, as noted in
  \cite[Section~2]{bucur:09:aa}, the arguments in
  \cite[pp.~466-7]{ladyzhenskaya:68:lqe} imply that $\psi$ is H\"older
  continuous on $\overline \Omega$. Also, by \cite{tolksdorf:84:rqe},
  $\nabla \psi$ is H\"older continuous inside $\Omega$.
\end{proof}

To prove Theorem~\ref{th:rp2}, we first reduce to the case that
$\Omega$ is connected. For, suppose Theorem~\ref{th:rp2} holds for
connected domains, and that $\Omega \neq D_2$ is not connected. There
are two possibilities: either $\lambda_2 (\Omega) = \lambda_2
(\widetilde \Omega)$ for some c.c.~$\widetilde \Omega$
of $\Omega$, or else there exist c.c.s $\Omega'$,
$\Omega''$ such that $\lambda_1 (\Omega) = \lambda_1 (\Omega')$,
$\lambda_2 (\Omega) = \lambda_1 (\Omega')$. In the former case, if we
let $\widetilde D_2$ be a scaled down version of $D_2$ with
$|\widetilde D_2| = |\widetilde \Omega|$, since $\widetilde \Omega$ is
connected we may apply Theorem~\ref{th:rp2} to get
$\lambda_2(\widetilde \Omega) > \lambda_2 (\widetilde D_2) \geq
\lambda_2 (D_2)$, where for the last step we have used
Lemma~\ref{lemma:ball}. In the latter case, let $B'$, $B''$ be balls
having the same volume as $\Omega'$, $\Omega''$, respectively. Then by
the Faber-Krahn inequality \cite[Theorem~1.1]{bucur:09:aa}, $\lambda_2
(\Omega) \geq \max \{\lambda_1(\Omega'), \lambda_1(\Omega'')\} \geq
\max \{\lambda_1(B'), \lambda_1 (B'')\}$, and the latter maximum is
minimised when $\lambda_1 (B') = \lambda_1 (B'') = \lambda_2(D_2)$.
Finally, if $\lambda_2 (\Omega) = \lambda_2 (D_2)$ then equality
everywhere in the above argument implies $|\Omega'| = |\Omega''| =
M/2$ (also using strict monotonicity in Lemma~\ref{lemma:ball}) and
sharpness of the Faber-Krahn inequality
\cite[Theorem~1.1]{bucur:09:aa} implies $\Omega' = B'$, $\Omega'' =
B''$; that is, $\Omega = D_2$.

So now suppose $\Omega$ is connected, and let $\psi \in
W^{1,p}(\Omega) \cap C(\overline \Omega)$ be any eigenfunction
associated with $\lambda_2 (\Omega)$. Since $\psi$ must change sign in
$\Omega$, the nodal domains $\Omega^+ := \{x \in \Omega: \psi(x)>0\}$
and $\Omega^- := \{x \in \Omega: \psi(x)<0\}$ are both nonempty and
open. Set $\psi^+:= \max\{\psi, 0\}$, $\psi^- := \max\{-\psi, 0\}$;
then we have $\psi^+, \psi^- \in W^{1,p} (\Omega) \cap C(\overline
\Omega)$, and
\begin{displaymath}
  \nabla \psi^+ =
  \begin{cases}
    \nabla \psi \qquad &\text{if $\psi>0$}\\
    0 &\text{if $\psi \leq 0$},
  \end{cases}
\end{displaymath}
with an analogous formula for $\nabla \psi^-$ (see
\cite[Lemma~7.6]{gilbarg:83:pde}).

Let $B^+$, $B^-$ be balls having the same volume as $\Omega^+$,
$\Omega^-$ respectively. We will show that $\lambda_2(\Omega) > \max
\{\lambda_1(B^+), \lambda_1(B^-)\}$. By Lemma~\ref{lemma:ball} this
maximum is minimal when $B^+ = B^-$ and $\lambda_1 (B^+)= \lambda_1
(B^-) = \lambda_2 (D_2)$. Without loss of generality we only consider
$\Omega^+$.  Let $\partial_e \Omega^+ := \partial \Omega^+ \cap
\partial \Omega$ and $\partial_i \Omega^+ := \partial \Omega^+ \cap
\Omega = \partial \Omega^+ \setminus \partial_e \Omega^+$ denote the
exterior and interior parts of the boundary of $\Omega^+$,
respectively (note that $\partial_i \Omega^+$ will not be closed). We
first show that a piece of $\partial_i \Omega^+$ must be smooth.

\begin{lemma}
  \label{lemma:smoothpiece}
  There exist $x_0 \in \Omega$ and $r>0$ such that $\psi(x_0) = 0$,
  $B(x_0, r) \subset \! \subset \Omega$, $\nabla \psi(x) \neq 0$ for
  all $x \in B(x_0, r)$, $\psi \in C^\infty (B(x_0,r))$ and $\{ x \in
  B(x_0, r): \psi(x) = 0\}$ is a surface of class $C^\infty$.
\end{lemma}

\begin{proof}
  We first show we can find $x_0 \in \partial_i \Omega^+$ with $\nabla
  \psi(x) \neq 0$ in a neighbourhood of $x_0$. Choose any $x \in
  \Omega^+$ close to $\partial_i \Omega^+$ and let $\delta_0 := \inf
  \{ \delta > 0: \partial B(x, \delta) \cap \partial_i \Omega^+ \neq
  \emptyset \}$. Then $B(x, \delta_0) \subset \Omega^+$ but there
  exists $x_0 \in \partial B(x, \delta_0) \cap \partial_i \Omega^+$.

  We now apply a version of Hopf's Lemma for the $p$-Laplacian due to
  V{\'a}zquez. Since $\psi(x_0) = 0$, $\psi(x)>0$ in $B(x, \delta_0)$
  and $\psi \in C^1(\overline{B(x, \delta_0)})$, by
  \cite[Theorem~5]{vazquez:84:smp} we have $\frac{\partial
    \psi}{\partial \nu_B} (x_0) < 0$, where $\nu_B$ is the outer unit
  normal to $B(x, \delta_0)$.  Hence $\nabla \psi (x_0) \neq 0$, and
  so by continuity of $\nabla \psi$ there exists a neighbourhood $V_0$
  of $x_0$ and $m>0$ such that $|\nabla \psi(x)| \geq m$ for all $x
  \in V_0$. In particular, inside $V_0$ we may write $-\Delta_p \psi =
  -\divergence (a(x) \nabla \psi)$, where $a(x) = |\nabla
  \psi(x)|^{p-2} \geq m^{p-2}>0$.  Since $\psi \in
  C^1(\overline{V_0})$ is an eigenfuction of the operator
  $-\divergence (a(x) \nabla u)$, a standard bootstrapping argument
  using elliptic regularity theory yields $\psi \in C^\infty (V_0)$.
  By the implicit function theorem it follows that the level surface
  $\{ \psi = 0 \}$ is locally the graph of a $C^\infty$ function
  inside $V_0$.
\end{proof}

Fix $x_0$ and $r$ as in the lemma and set $\Gamma := \partial_i
\Omega^+ \cap B(x_0, r/2)$ smooth; then the surface measure
$\sigma(\Gamma)>0$. We will impose Robin boundary conditions on
$\Gamma$, strictly lowering the first eigenvalue of a suitable
variational problem on $\Omega^+$. To that end set $V_0:= \{ \varphi
\in W^{1,p}(\Omega^+) \cap C(\overline{\Omega^+}): \varphi = 0 \text{
  on } \partial_i \Omega^+ \setminus \Gamma \}$, for $\varphi \in V_0$
set
\begin{equation}
  \label{eq:q}
  Q_p (\varphi) := \frac{\int_{\Omega^+} |\nabla \varphi|^p\,dx + 
    \int_{\partial_e \Omega^+ \cup \Gamma}  \alpha |\varphi|^p
    \,d\sigma}{\int_{\Omega^+} |\varphi|^p\,dx}
\end{equation}
and let
\begin{equation}
  \label{eq:kappa}
  \kappa (\Omega^+) := \inf_{\varphi \in V_0} Q_p (\varphi)
\end{equation}
We may characterise $\lambda_2(\Omega)$ as follows. In an abuse of
notation we will not distinguish between $\psi^+$ on $\Omega$ and
$\psi^+|_{\Omega^+}$.

\begin{lemma}
  \label{lemma:pchar}
  We have $\psi^+ \in V_0$ and
  \begin{equation}
    \label{eq:pq}
    \lambda_2 (\Omega) = Q_p(\psi) = Q_p (\psi^+) \equiv 
    \frac{\int_{\Omega^+} |\nabla \psi^+|^p\,dx + \int_{\partial_e 
        \Omega^+}\alpha|\psi^+|^p\,d\sigma}{\int_{\Omega^+} |\psi^+|^p\,dx}.
  \end{equation}
\end{lemma}

\begin{proof}
  We already know $\psi^+ \in V_0$, since $\psi^+ \in W^{1,p}(\Omega)
  \cap C(\overline \Omega)$ is zero on $\partial_i \Omega^+$. To
  obtain \eqref{eq:pq}, choose $\psi^+ $ as a test function in the
  characterisation \eqref{eq:eigenmeaning} of $\lambda_2(\Omega)$.
  Then $|\nabla \psi|^{p-2} \nabla \psi \cdot \nabla \psi^+ = |\nabla
  \psi^+|^p$ in $\Omega$ and $|\psi|^{p-2} \psi \, \psi^+ =
  |\psi^+|^p$ pointwise in $\overline \Omega$. Since $\| \psi^+ \|_p^p
  \neq 0$,
  \begin{equation}
    \label{eq:usualchar}
    \lambda_2 (\Omega) = \frac{\int_\Omega |\nabla \psi^+|^p\,dx
      + \int_{\partial \Omega} \alpha |\psi^+|^p\,d\sigma}
    {\int_\Omega |\psi^+|^p\,dx}.
  \end{equation}
  Now \eqref{eq:pq} follows since $\{x \in \Omega: \psi^+(x) \neq
  0\},\, \{x \in \Omega: \nabla \psi^+(x) \neq 0\} \subset \Omega^+$,
  and the boundary integrand $\alpha |\psi^+|^p$ in
  \eqref{eq:usualchar} is nonzero only on $\partial_e \Omega^+$.
  Finally, $Q_p (\psi) = Q_p (\psi^+)$ is obvious since $\psi \equiv
  \psi^+$ on $\Omega^+ \cup \partial_e \Omega^+$.
\end{proof}

\begin{lemma}
  \label{lemma:kappabound}
  $\lambda_2 (\Omega) > \kappa (\Omega^+)$.
\end{lemma}

\begin{proof}
  It is immediate from Lemma~\ref{lemma:pchar} and \eqref{eq:kappa}
  that $\lambda_2 (\Omega) \geq \kappa (\Omega^+)$.  Suppose for a
  contradiction that we have equality. Then since $\lambda_2 (\Omega)$
  and $\psi$ satisfy \eqref{eq:kappa}, we may also characterise them
  by
  \begin{displaymath}
    \begin{split}
      \int_{\Omega^+}|\nabla \psi|^{p-2} \nabla\psi\cdot\nabla\varphi\,dx
      &+ \int_{\partial_e \Omega^+ \cup \Gamma}\!\!
      \alpha |\psi|^{p-2}\psi \varphi\,d\sigma\\ &= \lambda_2(\Omega)
      \int_{\Omega^+}|\psi|^{p-2} \psi \varphi\,dx
    \end{split}
  \end{displaymath}
  for all $\varphi \in V_0$. (This can be seen, for example, by solving
  \begin{displaymath}
    \frac{d}{dt}\Bigl(\frac{\int_{\Omega^+}|\nabla(\psi-t\varphi)|^p\,dx
      +\int_{\partial_e \Omega^+ \cup \Gamma} \alpha|\psi-t \varphi|^p\,
      d\sigma}{\int_{\Omega^+}|\psi-t\varphi|^p\,dx}\Bigr)\Big|_{t=0} = 0,
  \end{displaymath}
  where $t \in \R$ and $\varphi \in V_0$.)

  Now recall $\partial_i \Omega^+$ is smooth in an open neighbourhood
  $B(x_0,r)$ of $\overline \Gamma$. We can choose an open set $U
  \subset \Omega^+$ Lipschitz with $U \subset\!\subset B(x_0, r)$ and
  such that $\Gamma \subset \partial U$. Then we may extend any
  $\varphi \in C_c^\infty (U \cup \Gamma)$ by zero to obtain an
  element of $V_0$, and so
  \begin{displaymath}
    \int_U |\nabla \psi|^{p-2} \nabla\psi\cdot\nabla\varphi \, dx+
    \int_{\Gamma} \alpha |\psi|^{p-2}
    \psi \varphi\,d\sigma = \lambda_2(\Omega) \int_U |\psi|^{p-2} 
    \psi \varphi\,dx.
  \end{displaymath}
  for any $\varphi \in C_c^\infty (U \cup \Gamma)$. Also, since $\psi
  \in C^\infty (\overline U)$ by Lemma~\ref{lemma:smoothpiece}, we see
  $-\Delta_p \psi = \lambda_2(\Omega) |\psi|^{p-2} \psi$ pointwise in
  $U$. Multiplying through by $\varphi \in C_c^\infty (U \cup
  \Gamma)$, a simple calculation gives
  \begin{displaymath}
    \int_U |\nabla \psi|^{p-2} \nabla\psi\cdot\nabla\varphi -
    \divergence(|\nabla \psi|^{p-2} \varphi \nabla \psi)\,dx
    = \lambda_2(\Omega) \int_U |\psi|^{p-2} \psi \varphi\,dx,
  \end{displaymath}
  which is valid since $\psi \in C^\infty (\overline U)$.  Applying
  the divergence theorem on $U$ (see for example
  \cite[Section~5.8]{evans:92:mt}) and comparing the above identities,
  \begin{displaymath}
    \int_{\Gamma} \alpha |\psi|^{p-2} \psi \varphi \,d\sigma
    = - \int_{\Gamma} |\nabla \psi|^{p-2}\frac{\partial\psi}
    {\partial \nu} \varphi\,d\sigma
  \end{displaymath}
  for all $\varphi \in C_c^\infty (U \cup \Gamma)$, where $\nu$ is the
  outward pointing unit normal to $U$ (equivalently, $\Omega^+$) on
  $\Gamma$.  Since $C_c^\infty (U \cup \Gamma)$ is dense in $L^q
  (\Gamma)$ for all $1<q<\infty$, it follows that $\psi \in
  C^1(\overline U)$ satisfies the boundary condition $|\nabla
  \psi|^{p-2}\frac{\partial \psi}{\partial \nu} + \alpha
  |\psi|^{p-2}\psi = 0$ pointwise in $\Gamma$. But we know $\psi = 0$
  on $\Gamma$, while by Hopf's Lemma \cite[Theorem~5]{vazquez:84:smp}
  applied to $U$ and $\psi \in C^1(\overline U)$, we have
  $\frac{\partial \psi}{\partial \nu} > 0$ (and $|\nabla \psi|>0$) on
  $\Gamma$, a contradiction.
\end{proof}

We will now construct a sequence of smooth domains $U_n$ approximating
$\Omega^+$ from the outside, in order to overcome the possible lack of
overall smoothness of $\partial \Omega^+$. As in
\cite[Section~3]{kennedy:09:lr2}, we attach a ``strip'' near $\partial
\Omega$ to $\Omega^+$ to avoid the points where $\partial_e \Omega^+$
and $\partial_i \Omega^+$ meet. Fix $n \geq 1$ and set $S_n := \{ x
\in \Omega: \dist (x, \partial \Omega) < \delta \}$, where $\delta =
\delta(n)$ is chosen such that $|S_n| < 1/(2n)$. By
\cite[Theorem~V.20]{edmunds:87:st} we can approximate $\Omega^+ \cup
S_n$ from the outside by a smooth domain $U_n$ as follows. Let $\Omega
\supset U_n \supset \Omega^+ \cup S_n$ be such that $\partial U_n =
\partial \Omega \cup \Gamma_n$, where $\Gamma_n \subset\!\subset
\Omega$ is $C^\infty$ and $|U_n \setminus (\Omega^+ \cup S_n)| <
1/(2n)$. We also impose the condition that $\Gamma \subset \Gamma_n$,
which we can do since $\partial_i \Omega^+$ is $C^\infty$ in an open
neighbourhood $B(x_0, r) \subset \Omega$ containing $\overline
\Gamma$. Then for any $n \geq 1$, $U_n$ is Lipschitz, $|U_n \setminus
\Omega^+| < 1/n$, and since $B(x_0, r) \subset \! \subset \Omega$,
without loss of generality $\dist(\overline U_n \setminus \Omega^+,
\Gamma) > 0$ as well. (See Figure~\ref{fig:domains}.)
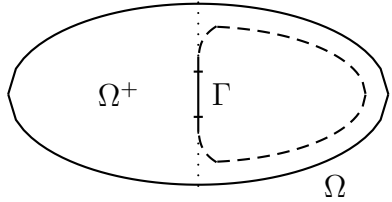
\begin{figure}[ht]
  \centering
  \begin{pspicture}*(-3,-1.4)(3,1.3)
    \psplot{-2.5}{2.5}{1.2 1 0.16 x 2 exp mul sub sqrt mul}%
    \psplot{-2.5}{2.5}{1.2 1 0.16 x 2 exp mul sub sqrt neg mul}%
    \psplot[linestyle=dashed]{0.25}{2.20001}{0.9 1 0.2066 x 2 
      exp mul sub sqrt mul}%
    \psplot[linestyle=dashed]{0.25}{2.2}{0.9 1 0.2066 x 2 
      exp mul sub sqrt mul neg}%
    \psline[linestyle=dotted](0,-1.23)(0,-0.5)%
    \psline[linestyle=dotted](0,0.5)(0,1.23)%
    \psline(0,-0.5)(0,0.5)%
    \psline(-0.06,0.3)(0.06,0.3)%
    \psline(-0.06,-0.3)(0.06,-0.3)%
    \uput[l](-0.6,0){$\Omega^+$}%
    \uput[r](0,0){$\Gamma$}%
    \psplot[linestyle=dashed]{0.01}{0.24}{0.15 x 0.2 exp add}%
    \psplot[linestyle=dashed]{0.01}{0.24}{0.15 x 0.2 exp add neg}%
    \uput[d](1.8,-0.9){$\Omega$}%
  \end{pspicture}
  \caption{$\Omega^+$ and $U_n$. The dotted line represents
    $\partial_i \Omega^+$ and the dashed line $\Gamma_n = \partial U_n
    \cap \Omega$.}
  \label{fig:domains}
\end{figure}

In order to use the $U_n$, we need the following modification of the
standard result that if $U \subset \R^N$ is open, arbitrary then
functions in $W^{1,p}(U)$ vanishing continuously on $\partial U$ lie
in $W^{1,p}_0 (U)$ (cf.~\cite[Section~7.5]{gilbarg:83:pde}).

\begin{lemma}
  \label{lemma:contain}
  Let $\varphi \in V_0$ and fix $n \geq 1$. The function $\tilde
  \varphi: U_n \to \R$ given by $\tilde \varphi = \varphi$ in
  $\Omega^+$, $\tilde \varphi = 0$ in $U_n \setminus \Omega^+$
  lies in $W^{1,p} (U_n)$.
\end{lemma}

\begin{proof}
  Let $\varphi \in V_0$ and $\tilde \varphi$ be as in the statement of
  the lemma. Using the lattice properties of $V_0$ and $W^{1,p} (U_n)$
  (cf.~\cite[Lemma~7.6]{gilbarg:83:pde}) we may assume that $\varphi
  \geq 0$ in $\Omega^+$. For $\xi > 0$ let $\varphi_\xi := (\varphi -
  \xi)^+ \in V_0$. Then by continuity of $\varphi$, there exists an
  open neighbourhood $U = U(\varphi, \xi)$ of $\partial_i \Omega^+
  \setminus \Gamma$ such that $\varphi_\xi \equiv 0$ on $U \cap
  \overline{\Omega^+}$. Since the intersection of $U_n \setminus
  \Omega^+$ with $\overline{\Omega^+}$ is contained in $\partial_i
  \Omega^+ \setminus \Gamma$, adapting the argument in
  \cite[Th{\'e}or{\`e}me~IX.17]{brezis:83:af} we may certainly extend
  $\varphi_\xi$ by $0$ in $U_n \setminus \Omega^+$ to obtain a
  function $\tilde \varphi_\xi \in W^{1,p} (U_n)$. Since $\tilde
  \varphi_\xi \nearrow \tilde \varphi$ and
  \begin{displaymath}
    \nabla \tilde \varphi_\xi(x) \nearrow g(x):=
    \begin{cases}
      \nabla \varphi(x) & \qquad \text{if $x \in \Omega^+$}\\
      0 & \qquad \text{if $x \in U_n \setminus \Omega^+$}
    \end{cases}
  \end{displaymath}
  pointwise monotonically in $U_n$ as $\xi \to 0$, it follows easily
  that $g = \nabla \tilde \varphi$ and $\tilde \varphi \in
  W^{1,p}(U_n)$.
\end{proof}

For any $n \geq 1$ and $\varphi \in V_0$, using the extension $\tilde
\varphi \in W^{1,p} (U_n)$ of $\varphi$ in the representation
\begin{displaymath}
  \lambda_1 (U_n) = \inf_{\tilde \varphi \in W^{1,p}(U_n)}
  \frac{\int_{U_n} |\nabla \tilde \varphi|^p\,dx + \int_{\partial U_n} 
    \alpha |\tilde\varphi|^p\,d\sigma}{\int_{U_n} |\tilde\varphi|^p\,dx}
\end{displaymath}
we see $Q_p (\varphi) \geq \lambda_1(U_n)$ for every $\varphi \in
V_0$. Hence $\kappa(\Omega^+) \geq \lambda_1 (U_n)$ by
\eqref{eq:kappa}. Now let $B_n$ be a ball with $|B_n| = |U_n|$. By the
Faber-Krahn inequality \cite[Theorem~1.1]{bucur:09:aa}, $\lambda_1
(U_n) \geq \lambda_1 (B_n)$. As $n \to \infty$, $|U_n| \to |\Omega^+|$
and so $\lambda_1 (B_n) \to \lambda_1 (B^+)$ by
Lemma~\ref{lemma:ball}. We conclude that $\lambda_2 (\Omega) > \kappa
(\Omega^+) \geq \limsup_{n \to \infty} \lambda_1 (U_n) \geq \lambda_1
(B^+)$, which in light of our earlier comments completes the proof.

\section{On the higher eigenvalues of the Robin Laplacian}
\label{sec:rhigher}

From now on we will assume $p=2$ in \eqref{eq:robinproblem}. We will
consider the problem \eqref{eq:minproblem} for $k \geq 3$ fixed. In
contrast to the Dirichlet case, this is not one problem but a family
depending on the parameter $\alpha>0$. Here we will show that one
cannot in general find a solution to \eqref{eq:minproblem} independent
of $\alpha$ (alternatively, of the volume $M$).  Roughly speaking, for
large $\alpha$ we are close to the corresponding Dirichlet problem,
while for $\alpha$ close to $0$ (a Neumann problem), the domain $D_k$
consisting of the disjoint union of $k$ equal balls is in some sense a
minimiser.  We will denote by $B_m$ a ball of volume $m$, so that
$D_k$ is the disjoint union of $k$ copies of $B_{M/k}$, and $\lambda_k
(D_k, \alpha) = \lambda_1 (D_k, \alpha) = \lambda_1 (B_{M/k},
\alpha)$.

\begin{theorem}
  \label{th:robink}
  Let $p=2$ in \eqref{eq:robinproblem}.
  \begin{itemize}
  \item[(i)] Given any $\Omega \subset \R^N$ of volume $M$ satisfying
    Remark~\ref{rem:general}(ii) such that $\Omega \neq D_k$ in the
    sense of Remark~\ref{rem:general}(iii), there exists
    $\alpha_\Omega > 0$ possibly depending on $\Omega$ such that
    $\lambda_k (\Omega, \alpha) > \lambda_k (D_k, \alpha)$ for all
    $\alpha \in (0, \alpha_\Omega)$.
  \item[(ii)] There exist $N \geq 2$ and $k \geq 3$ for which, given
    $M>0$, there is no solution to \eqref{eq:minproblem} independent of
    $\alpha$; equivalently, there is no domain $D$ satisfying
    $\lambda_k (\Omega, \alpha) \geq \lambda_k (D, \alpha)$ for all
    $\alpha \in (0,\infty)$ and all $\Omega$.
  \item[(iii)] There exist $N \geq 2$ and $k \geq 3$ for which, given
    $\alpha > 0$, there is no solution to \eqref{eq:minproblem}
    independent of $M>0$.
  \end{itemize}
\end{theorem}

\begin{remark}
  \label{rem:higherexamples}
  (i) The conclusion of Theorem~\ref{th:robink}(ii) and (iii) holds
  whenever $D_k$ does not minimise the $k$th Dirichlet eigenvalue
  $\mu_k$. When $N=2$ this is true for all $k \geq 3$ (we prove this
  below) and when $N=3$ at least for $k=3$ (for the latter see
  \cite[Section~3]{bucur:00:3de}).

  (ii) It is easy to see (ii) and (iii) are equivalent assertions,
  since by making the homothety substitution $x \mapsto \alpha x$,
  \eqref{eq:robinproblem} is equivalent to the problem $-\Delta u =
  (\lambda/\alpha^2) u$ in $\alpha \Omega = \{\alpha x: x \in
  \Omega\}$, $\frac{\partial u}{\partial \nu} + u = 0$ on $\partial
  (\alpha \Omega)$.

  (iii) It is clear that any domain $\Omega$ with more than $k$
  connected components (c.c.s) cannot minimise $\lambda_k$ for any
  value of $\alpha$. However, the theorem makes a stronger statement
  than this and as a result the proof is somewhat more involved.
  Indeed, for some $k$, $N$, we can easily find a domain $\Omega_n$
  with any $n \geq 1$ c.c.s and $\alpha_\Omega< \infty$. (Just take
  $N=k=3$, so that for the ball $B$, $\alpha_B<\infty$. Shrink $B$
  slightly and add $n-1$ disjoint tiny balls to get $\Omega_n$.) Note
  that the Robin problem \eqref{eq:robinproblem} lacks many useful
  properties that the corresponding Dirichlet problem satisfies. For
  example, the domain monotonicity property fails; that is, $U \subset
  V$ does not necessarily imply $\lambda_k (U, \alpha) \geq \lambda_k
  (V, \alpha)$ (see \cite{payne:57:lb} or \cite{giorgi:05:mr} for a
  counterexample). Similarly, if $\lambda_k (U, \alpha) > \lambda_k
  (V, \alpha)$ holds for some $\alpha>0$, we cannot in general expect
  this for all $\alpha>0$.

  (iv) An examination of our proof shows that the conclusion of
  Theorem~\ref{th:robink}(i) holds for any domain $\Omega$ for which
  the Faber-Krahn inequality \cite[Theorem~1.1]{bucur:09:aa} and
  Theorem~\ref{th:rp2} hold.
\end{remark}

\begin{proof}[Proof of Theorem~\ref{th:robink}(i)]
  There are two cases to consider, depending on how many c.c.s
  $\Omega$ has.

  (i) Suppose first that $\Omega$ has at most $k-1$ c.c.s. If we set
  $\varepsilon := \min\, \{\lambda_2 (\widetilde \Omega, 0) :
  \widetilde \Omega \text{ is a c.c.~of } \Omega \}$, then
  $\varepsilon > 0$ by Lemma~\ref{lemma:con}. It follows from
  Lemma~\ref{lemma:cm}(i) that there exists $\tilde \alpha_\Omega > 0$
  such that
  \begin{displaymath}
    \max\, \{ \lambda_1 (\widetilde \Omega, \alpha) :
    \widetilde \Omega \text{ is a c.c.~of } \Omega \} < \varepsilon 
  \end{displaymath}
  for all $\alpha \in (0, \tilde \alpha_\Omega)$. For all such
  $\alpha$, by the pigeonhole principle at least one element of the
  set $\{ \lambda_m (\widetilde \Omega, \alpha) : m \geq 2, \widetilde
  \Omega \text{ is a c.c.~of } \Omega \}$ must be one of the first $k$
  eigenvalues of $\Omega$ (although precisely which $m$ and c.c.~may
  depend on $\alpha$). In particular, using Lemma~\ref{lemma:cm}(i),
  \begin{displaymath}
    \begin{split}
      \lambda_k (\Omega, \alpha) & \geq \inf \, \{ \lambda_m (\widetilde 
      \Omega, \alpha) : m \geq 2,\, \widetilde \Omega \text{ is a c.c.~of } 
      \Omega \}\\
      &\geq \inf \, \{\lambda_2 (\widetilde \Omega, 0) : \widetilde \Omega 
      \text{ is a c.c.~of } \Omega \} \geq \varepsilon
    \end{split}
  \end{displaymath}
  for all $\alpha \in (0, \tilde \alpha_\Omega)$. Since $\lambda_k
  (D_k, \alpha) = \lambda_1 (D_k, \alpha) \to 0$ as $\alpha \to 0$,
  there exists $0< \alpha_\Omega \leq \tilde \alpha_\Omega$ such that
  $\lambda_k (D_k, \alpha) < \varepsilon \leq \lambda_k (\Omega,
  \alpha)$ for all $\alpha \in (0, \alpha_\Omega)$.

  (ii) Now suppose $\Omega$ has at least $k$ c.c.s. We may write
  $\Omega$ as the disjoint union of $\Omega'$ and $\Omega''$, where
  $\Omega'$ has $j < \infty$ c.c.s and $|\Omega''| < M/k$ (if
  $\Omega'' = \emptyset$, then we declare $\lambda_1 (\Omega'',
  \alpha) = \infty$ for all $\alpha>0$).  Consider all possible open
  subdomains $\Omega_i$ of $\Omega'$, where $\Omega_i$ consists of
  $l_i \leq k-1$ c.c.s of $\Omega'$ (thus there are fewer than $2^j$
  possible choices of $\Omega_i$). For each $i$, let $D_{k,i}$ denote
  a scaled down version of $D_k$ such that $|D_{k,i}| = |\Omega_i|$.
  Then by case (i) and Lemma~\ref{lemma:ball}, there exists $\alpha_i
  := \alpha_{\Omega_i}$ such that
  \begin{equation}
    \label{eq:iforii}
    \lambda_k(\Omega_i,\alpha)>\lambda_k(D_{k,i},\alpha)\geq
    \lambda_k(D_k,\alpha)
  \end{equation}
  for all $\alpha \in (0,\alpha_i)$.

  Set $\alpha_\Omega:= \min_i \alpha_i > 0$ and fix $\alpha \in (0,
  \alpha_\Omega)$. We will show $\lambda_k (\Omega, \alpha) \geq
  \lambda_k (D_k, \alpha)$ with equality only if $\Omega = D_k$ in
  the sense of Remark~\ref{rem:general}(iii).

  First suppose $\lambda_1(\Omega'',\alpha)\leq \lambda_k(\Omega,
  \alpha)$. Then by the Faber-Krahn inequality
  \cite[Theorem~1.1]{bucur:09:aa} and Lemma~\ref{lemma:ball}
  \begin{equation}
    \label{eq:inftycase}
    \lambda_k(\Omega,\alpha) \geq \lambda_1(\Omega'',\alpha) \geq
    \lambda_1(B_{M/k}) = \lambda_k (D_k, \alpha).
  \end{equation}
  Since $|\Omega''|<M/k$, Lemma~\ref{lemma:ball} implies that the
  second inequality in \eqref{eq:inftycase} must be strict.

  So assume now that $\lambda_1 (\Omega'', \alpha) > \lambda_k
  (\Omega, \alpha)$. There are two subcases to consider. First, if
  there are only $l < k$ c.c.s $\Omega_1, \ldots, \Omega_l$ of
  $\Omega'$ whose first eigenvalue is smaller than $\lambda_k (\Omega,
  \alpha)$, then setting $\widehat \Omega$ to be the disjoint union of
  $\Omega_1, \ldots, \Omega_l$, by \eqref{eq:iforii} we have
  \begin{displaymath}
    \lambda_k (\Omega, \alpha) = \lambda_k (\widehat \Omega, \alpha)
    > \lambda_k (D_k, \alpha)
  \end{displaymath}
  by choice of $\alpha_\Omega$ and $\alpha<\alpha_\Omega$. Finally,
  suppose there are at least $k$ c.c.s $\Omega_i$ of $\Omega'$ such
  that $\lambda_1 (\Omega_i, \alpha) \leq \lambda_k (\Omega, \alpha)$
  for all $i$. Then $\lambda_k (\Omega, \alpha) = \max_{1 \leq i \leq
    k} \lambda_1 (\Omega_i, \alpha)$. For each $i$ let $B_i$ be a ball
  with $|B_i| = |\Omega_i|$. By the Faber-Krahn inequality $\lambda_1
  (\Omega_i, \alpha) \geq \lambda_1 (B_i)$ for all $i$ and thus
  \begin{equation}
    \label{eq:allballs}
    \lambda_k (\Omega, \alpha) \geq \max_i \lambda_1(B_i, \alpha)
    \geq \lambda_1 (B_{M/k}, \alpha) = \lambda_k (D_k, \alpha),
  \end{equation}
  where the second inequality in \eqref{eq:allballs} follows easily
  from Lemma~\ref{lemma:ball} using $\sum_i |B_i| \leq |\Omega|$.  If
  there is equality in \eqref{eq:allballs}, then for every $1 \leq i
  \leq k$, $\lambda_1 (\Omega_i, \alpha) = \lambda_1 (B_i, \alpha) =
  \lambda_1 (B_{M/k}, \alpha)$ and so $\Omega_i = B_i = B_{M/k}$ using
  sharpness of the Faber-Krahn inequality
  \cite[Theorem~1.1]{bucur:09:aa} and Lemma~\ref{lemma:ball},
  respectively. In this case $|\Omega_i| = M/k$ and so $\Omega$ must
  consist of $k$ copies of $\Omega_i = B_{M/k}$, so $\Omega = D_k$.
\end{proof}

In order to complete the proof of the theorem and our claim in
Remark~\ref{rem:higherexamples}(i), we will use the following lemma.
Recall $\mu_k (\Omega)$ denotes the $k$th eigenvalue of the Dirichlet
Laplacian (with $p=2$) on $\Omega$.

\begin{lemma}
  \label{lemma:mukmin}
  Let $N=2$ and fix $k \geq 3$. The domain $D_k$ does not minimise
  $\mu_k(\Omega)$ amongst all bounded Lipschitz domains in $\R^2$ of
  given volume.
\end{lemma}

\begin{proof}
  The proof is by an easy induction argument, using results from
  \cite{wolf:94:re}. First note that $D_k$ does not even minimise
  $\mu_k$ amongst all disjoint unions of balls if $3 \leq k \leq 17$
  (see \cite[Section~8]{wolf:94:re}).

  Now fix $k \geq 4$. We will show that if $D_{k+1}$ minimises
  $\mu_{k+1}$, then $D_j$ must minimise $\mu_j$ for some $3 \leq j
  \leq k$. For, arguing as in \cite[Theorem~8.1]{wolf:94:re},
  $D_{k+1}$ may be written as the disjoint union of open sets $U$ and
  $V$, say, where $U$ minimises $\mu_j$ and $V$ minimises
  $\mu_{k-j+1}$ (both appropriately scaled) for some integer $j$
  between $1$ and $k/2$. Now $U$ and $V$ must both be disjoint unions
  of equal balls, and since the minimiser of $\mu_j$ can have at most
  $j$ c.c.s the only possibility is that $U = D_j$ and $V = D_{k-j+1}$
  (both rescaled). Since $k \geq 4$, at least one of $j$, $k-j+1$ must
  be at least $3$. Noting that the Dirichlet minimiser is independent
  of the volume of the domain, our claim follows.
\end{proof}

\begin{proof}[Proof of Theorem~\ref{th:robink}(ii) and 
  Remark~\ref{rem:higherexamples}(i)]
  Suppose that $D_k$ is not the minimiser of $\mu_k$, which is true if
  $N=2$ and $k \geq 3$ or $N=k=3$. Then there exists a Lipschitz
  domain $V$ such that $\mu_k(V) < \mu_k(D_k)$. By
  Lemma~\ref{lemma:cm}(ii) and (iii), we have $\lambda_k(V, \alpha) <
  \mu_k(V)$ and $\lambda_k (D_k, \alpha) = \lambda_1(D_k, \alpha) \to
  \mu_1(D_k) = \mu_k(D_k)$ as $\alpha \to \infty$. Using continuity,
  it follows that for $\alpha$ sufficiently large, $\lambda_k (V,
  \alpha) < \mu_k (D_k, \alpha)$. Hence $D_k$ does not minimise
  $\lambda_k$ for all $\alpha \in (0, \infty)$. However, if $U \neq
  D_k$ is any (Lipschitz) domain which minimises $\lambda_k$ for some
  $\tilde \alpha \in (0,\infty)$, then by part (i) $\lambda_k (U,
  \alpha) > \lambda_k (D_k, \alpha)$ for $\alpha < \tilde \alpha$
  sufficiently small. Hence for such $N$ and $k$ no minimiser can
  exist for all $\alpha>0$.
\end{proof}

\section{On the higher eigenvalues of the Wentzell Laplacian}
\label{sec:whigher}

Here we will study the Laplacian with generalised Wentzell boundary
conditions \eqref{eq:wentzellproblem}. This problem has been
extensively studied in recent years; see for example
\cite{goldstein:06:gbc,mugnolo:06:dfw} and the references therein. We
will denote by $\Lambda_k = \Lambda_k (\Omega, \beta, \gamma)$ the
$k$th eigenvalue, with repeated eigenvalues counted according to their
multiplicity. It was proved in \cite{kennedy:08:wfk} that if $\Omega
\subset \R^N$ is a bounded Lipschitz domain, then
\begin{equation}
  \label{eq:wentzell1}
  \Lambda_1 (\Omega, \beta, \gamma) \geq \Lambda_1 (B, \beta, \gamma)
\end{equation}
for all $\beta, \gamma > 0$. (As before $B$ is a ball having the same
volume $M$ as $\Omega$.) Moreover, the inequality is sharp if $\Omega$
is of class $C^2$. Note that combining the improved sharpness result
in \cite{bucur:09:aa} for Robin problems with the method in
\cite{kennedy:08:wfk}, we immediately get sharpness of the Wentzell
inequality \eqref{eq:wentzell1} for all bounded Lipschitz domains. We
will prove the following results which basically say that the
minimisation problems for the Robin and Wentzell Laplacians are
essentially the same.

\begin{theorem}
  \label{th:wentzellk}
  Let $\beta, \gamma > 0$ and $k \geq 2$ be fixed, let $D \subset
  \R^N$ be a bounded Lipschitz domain, and let $D_k \subset \R^N$ be
  as in Section~\ref{sec:rhigher}.
  \begin{itemize}
  \item[(i)] Suppose that for every bounded Lipschitz $\Omega \subset
    \R^N$ we have
    \begin{equation}
      \label{eq:robininwentzell}
      \lambda_k(D, \alpha) \leq \lambda_k(\Omega, \alpha)
    \end{equation}
    for all $\alpha \in (0, \gamma/\beta)$. Then
    \begin{equation}
      \label{eq:wentzellinwentzell}
      \Lambda_k (D, \beta, \gamma) \leq \Lambda_k (\Omega, \beta, \gamma)
    \end{equation}
    for all such $\Omega$. Conversely, if
    \eqref{eq:wentzellinwentzell} holds, then
    \eqref{eq:robininwentzell} holds for some $\alpha \in (0,
    \gamma/\beta)$.
  \item[(ii)] If \eqref{eq:robininwentzell} is sharp for all $\alpha
    \in (0, \gamma/\beta)$, then so is \eqref{eq:wentzellinwentzell}
    for this $\beta,\gamma$. If \eqref{eq:wentzellinwentzell} is
    sharp, then \eqref{eq:robininwentzell} holds and is sharp for some
    $\alpha \in (0, \gamma/\beta)$.
  \item[(iii)] Suppose $\Omega \subset \R^N$ is bounded, Lipschitz.
    There exists $\alpha_\Omega > 0$ possibly depending $\Omega$ such
    that $\Lambda_k (\Omega, \beta, \gamma) > \Lambda_k (D_k, \beta,
    \gamma)$ for all $\beta, \gamma$ with $\gamma/\beta <
    \alpha_\Omega$.
  \item[(iv)] If for some $k$ and $N$ the conclusion of
    Theorem~\ref{th:robink} holds, then there does not exist $D
    \subset \R^N$ bounded, Lipschitz such that $\Lambda_k (\Omega,
    \beta, \gamma) \geq \Lambda_k (D, \beta, \gamma)$ for all such
    $\Omega$ and all $\beta, \gamma > 0$.
  \item[(v)] For any bounded, Lipschitz $\Omega \subset \R^N$ and any
    $\beta, \gamma>0$, we have $\Lambda_2 (\Omega, \beta, \gamma) \geq
    \Lambda_2 (D_2, \beta, \gamma)$, with equality if and only if
    $\Omega = D_2$.
  \end{itemize}
\end{theorem}

In order to prove the theorem we will need some preliminary results.
In what follows we will assume that $\beta, \gamma > 0$ and $k \geq 2$
are fixed, and $\Omega \subset \R^N$ is a fixed bounded Lipschitz
domain.  We start with an elementary identification which is the key
to the approach.

\begin{lemma}
  \label{lemma:wentzellid}
  Let $k \geq 1$ and $\alpha := (\gamma - \Lambda_k(\Omega, \beta,
  \gamma))/\beta \in \R$. Then
  \begin{equation}
    \label{eq:wentzellid}
    \Lambda_k (\Omega, \beta, \gamma) = \lambda_k (\Omega, \alpha).
  \end{equation}
\end{lemma}

\begin{proof}
  Consider the family of curves $g_n: \R \to \R$, $g_n(\alpha):=
  (\gamma-\lambda_n (\Omega, \alpha))/\beta$, $n \geq 1$, where we
  allow multiplicities in counting the $\lambda_n$ (thus if $\lambda_n
  (\Omega, \tilde \alpha) = \lambda_{n+1} (\Omega, \tilde \alpha)$ for
  some $\tilde \alpha \in \R$, then $g_n(\tilde \alpha) =
  g_{n+1}(\tilde \alpha)$).

  We know that the set of Wentzell eigenvalues $\{\Lambda_k: k \geq
  1\}$ is in one-to-one correspondence with the set of fixed points
  $\{ \alpha \in \R: g_n(\alpha) = \alpha \text{\ for some\ } n\}$,
  via the identification as in \cite[Proposition~3.3]{kennedy:08:wfk}
  (see also Remark~3.6(i) there). In particular, we know that
  $\Lambda_k (\Omega, \beta, \gamma) = \lambda_n (\Omega, \alpha)$
  with $\alpha = (\gamma - \Lambda_k)/\beta$ for some $n \geq 1$; we
  have to show $n = k$.

  Now by Lemma~\ref{lemma:cm}(i) each curve $g_n$ is a continuous and
  monotonically decreasing function of $\alpha$. In particular for
  each $n$ there will be exactly one fixed point $\alpha_n \in \R$ for
  which $g_n(\alpha_n)=\alpha_n$. Moreover, by definition $g_n(\alpha)
  \leq g_m(\alpha)$ whenever $n \geq m$ and hence $\alpha_n \leq
  \alpha_m$ if $n \geq m$. It follows inductively that $\lambda_n
  (\Omega, \alpha_n) = \gamma - \alpha_n \beta$ is the $n$th Wentzell
  eigenvalue $\Lambda_n (\Omega, \beta, \gamma)$ for all $n \geq 1$.
\end{proof}

Note that we have $0 < \Lambda_1 (\Omega, \beta, \gamma) = \gamma -
\alpha \beta$ for some $\alpha > 0$ (see
\cite[Remark~5.2]{kennedy:08:wfk}).  In particular, we obtain the
bound $\Lambda_1 (\Omega, \beta, \gamma) < \gamma$ always,
\emph{independent of the volume of $\Omega$}. This yields the
following result, which obviously remains true if we replace $D_k$ by
any domain $\Omega$ having at least $k$ c.c.s.

\begin{lemma}
  \label{lemma:dk}
  We have $\Lambda_k(D_k, \beta, \gamma) < \gamma$ for all $k \geq 1$.
\end{lemma}

\begin{proof}
  As in Section~\ref{sec:rhigher}, we write $D_k$ as the disjoint
  union of $k$ balls $B_{M/k}$. Then $\Lambda_k(D_k, \beta, \gamma) =
  \Lambda_1(B_{M/k}, \beta, \gamma) < \gamma$.
\end{proof}

We are now in a position to give the proof of
Theorem~\ref{th:wentzellk}. Since $\beta, \gamma$ are fixed we will
write $\Lambda_k (\Omega, \beta, \gamma) = \Lambda_k (\Omega)$ if
there is no danger of confusion. The following lemma contains the core
of the argument.

\begin{lemma}
  \label{lemma:core}
  Let $\beta, \gamma > 0$ be given and $U, V \subset \R^N$ bounded,
  Lipschitz.
  \begin{itemize}
  \item[(i)] If $\Lambda_k (U) < \gamma$, then for $\alpha:= (\gamma -
    \Lambda_k (U))/\beta$,
    \begin{equation}
      \label{eq:laineq}
      \lambda_k (U, \alpha) \geq \lambda_k (V, \alpha)
    \end{equation}
    implies
    \begin{equation}
      \label{eq:biglaineq}
      \Lambda_k (U) \geq \Lambda_k (V).
    \end{equation}
    If the equality in \eqref{eq:laineq} is strict, then it is also
    strict in \eqref{eq:biglaineq}.
  \item[(ii)] Suppose $\Lambda_k (V) < \gamma$ and let $\alpha:=
    (\gamma - \Lambda_k(V))/\beta$. If \eqref{eq:biglaineq} holds
    (resp.~is strict), then \eqref{eq:laineq} holds (resp.~is strict)
    for this $\alpha$.
  \end{itemize}
\end{lemma}

\begin{proof}
  (i) Suppose \eqref{eq:laineq} holds but \eqref{eq:biglaineq} fails.
  Using Lemma~\ref{lemma:wentzellid} and \eqref{eq:laineq}
  respectively,
  \begin{displaymath}
    \begin{split}
      \Lambda_k (U) &= \lambda_k (U, \frac{\gamma - \Lambda_k(U)}{\beta})\\
      &\geq \lambda_k (V, \frac{\gamma - \Lambda_k(U)}{\beta})
      \geq \lambda_k (V, \frac{\gamma - \Lambda_k(V)}{\beta})
      =\Lambda_k (V),
    \end{split}
  \end{displaymath}
  where the second inequality follows from Lemma~\ref{lemma:cm}(i)
  since $\gamma - \Lambda_k (U) \geq \gamma - \Lambda_k (V)$ by the
  contradiction assumption. Hence $\Lambda_k (U) \geq \Lambda_k (V)$,
  contradicting the assumption that \eqref{eq:biglaineq} fails. Now
  suppose \eqref{eq:laineq} is strict and the contradiction assumption
  becomes $\Lambda_k (U) \leq \Lambda_k (V)$. Since the first
  inequality in the above line of reasoning is now strict, we still
  obtain a contradiction as nothing else changes. Hence we cannot have
  equality in \eqref{eq:biglaineq}.

  (ii) Now suppose that \eqref{eq:biglaineq} holds and that
  \eqref{eq:laineq} fails. Interchanging the roles of $U$ and $V$, we
  may argue essentially exactly as in (i) to obtain the desired
  conclusion (and do similarly for strictness).
\end{proof}

\begin{proof}[Proof of Theorem~\ref{th:wentzellk}]
  (i) Suppose $D$ satisfies \eqref{eq:robininwentzell}. Let
  $(\Omega_m)_{m \in \N}$ be a minimising sequence for $\Lambda_k$. By
  Lemma~\ref{lemma:dk}, we may assume $\Lambda_k (\Omega_m) < \gamma$
  for all $m$, so that $(\gamma - \Lambda_k (\Omega_m))/\beta \in (0,
  \gamma/\beta)$ and thus \eqref{eq:robininwentzell} holds for these
  values of $\alpha$. Fixing $m \in \N$, we may apply
  Lemma~\ref{lemma:core}(i) with $\Omega_m$ in place of $U$ and $D$ in
  place of $V$ to conclude $\Lambda_k (\Omega_m) \geq \Lambda_k (D)$.
  Since $(\Omega_m)_{m \in \N}$ was a minimising sequence, $D$ must
  minimise $\Lambda_k (\Omega)$. For the converse, suppose $D$
  satisfies \eqref{eq:wentzellinwentzell}.  Since $\Lambda_k (D) <
  \gamma$ by Lemma~\ref{lemma:dk}, it follows directly from
  Lemma~\ref{lemma:core}(ii) that $D$ satisfies
  \eqref{eq:robininwentzell} for $\alpha = (\gamma - \Lambda_k
  (D))/\beta$.

  (ii) Sharpness in both directions now follows immediately from
  strictness of the inequalities in Lemma~\ref{lemma:core}.

  (iii) Fix $\Omega \neq D_k$. By Theorem~\ref{th:robink}(i), there
  exists $\alpha_\Omega>0$ such that $\lambda_k(\Omega, \alpha) >
  \lambda_k (D_k, \alpha)$ for all $\alpha \in (0, \alpha_\Omega)$. If
  $\beta, \gamma$ are fixed with $\gamma/\beta < \alpha_\Omega$, then
  we have $\lambda_k (\Omega, \alpha) > \lambda_k (D_k, \alpha)$ for
  $\alpha = (\gamma - \Lambda_k(\Omega))/\beta$ in particular. Since
  also $\Lambda_k(D_k) < \gamma$ by Lemma~\ref{lemma:dk}, without loss
  of generality we may assume $\Lambda_k (\Omega) < \gamma$ (otherwise
  $\Lambda_k (\Omega) \geq \gamma > \Lambda_k (D)$ and we are done).
  But in this case it follows from Lemma~\ref{lemma:core}(i) (with
  $\Omega=U$) that $\Lambda_k (\Omega) > \Lambda_k (D_k)$ anyway.

  (iv) Let $k$ and $N$ be such that the conclusion of
  Theorem~\ref{th:robink}(ii) holds. By (iii) it suffices to show
  there exist $\beta, \gamma > 0$ and a domain $\Omega$ with
  $\Lambda_k (\Omega, \beta, \gamma) < \Lambda_k (D_k, \beta,
  \gamma)$. Choose $\Omega$ and $\alpha^*>0$ such that $\lambda_k
  (\Omega, \alpha^*) < \lambda_k (D_k, \alpha^*)$. Now we may write
  $\Lambda_k (D_k, \beta, \gamma) = \Lambda_1 (D_k, \beta, \gamma) =
  \gamma - \alpha \beta$, where $\alpha$ satisfies $(\gamma -
  \lambda_1(D_k, \alpha))\beta = \alpha$. Since $\lambda_1 (D_k,
  \alpha)$ is continuous and monotonic with respect to $\alpha$, an
  elementary argument shows that by fixing $\beta$ and varying
  $\gamma$, we may obtain every $\alpha > 0$ as a solution to $(\gamma
  - \lambda_1(D_k, \alpha))\beta = \alpha$ for some $\beta, \gamma>0$.
  Now choose $\beta, \gamma$ such that $\Lambda_k (D_k, \beta, \gamma)
  = \gamma - \alpha^* \beta$. For this $\beta, \gamma$, we may apply
  Lemma~\ref{lemma:core}(i) with $U = D_k$ and $V = \Omega$ to
  conclude $\Lambda_k (D_k, \beta, \gamma) > \Lambda_k (\Omega, \beta,
  \gamma)$.

  (v) This follows immediately from (i) and (ii) combined with
  Theorem~\ref{th:rp2}.
\end{proof}

\appendix

\section{Some basic eigenvalue properties}
\label{sec:auxiliary}

Here we collect some elementary but useful facts about the behaviour
of the eigenvalues of the Robin and Neumann Laplacians.

\begin{lemma}
  \label{lemma:cm}
  Suppose $\Omega \subset \R^N$ is a fixed Lipschitz domain and $p=2$.
  Then the following assertions are true.
  \begin{itemize}
  \item[(i)] Let $k \geq 1$. Then $\lambda_k (\Omega, \alpha)$ is
    continuous and monotonically increasing as a function of $\alpha
    \in \R$.
  \item[(ii)] For any $\alpha \geq 0$ and $k \geq 1$, we have
    $\lambda_k(\Omega, \alpha) < \mu_k(\Omega)$.
  \item[(iii)] $\lambda_1 (\Omega, \alpha) \to \mu_1(\Omega)$ as
    $\alpha \to \infty$.
  \end{itemize}
\end{lemma}

\begin{proof}
  Parts (i) and (ii) follow immediately from the minimax formula for
  the $k$th eigenvalue (see \cite[Section~VI.1]{courant:53:mmp}. Note
  that although \cite{courant:53:mmp} only deals with the case $N=2$,
  none of the relevant arguments depend on the dimension of the
  space). For part (iii), see for example \cite{giorgi:05:mr}.
\end{proof}

Our next lemma expresses in our notation the well-known fact that the
first Neumann eigenvalue of a connected domain is simple, with
constant functions the only eigenfunctions. We omit the proof (see
\cite[Problem~2.2]{gilbarg:83:pde}).

\begin{lemma}
  \label{lemma:con}
  Let $p=2$. If $\Omega$ is bounded, Lipschitz and connected, then
  $\lambda_2 (\Omega, 0) > 0$.
\end{lemma}

The following equally well-known result is true in general for the
$k$th eigenvalue of \eqref{eq:robinproblem} on any reasonably smooth
domain, although we only need this for the first eigenvalue of a ball.
A proof (for balls) can be found in \cite[Lemma~4.1]{bucur:09:aa}.

\begin{lemma}
  \label{lemma:ball}
  Suppose $1<p<\infty$. Let $B_m$ denote the ball of volume $m$,
  centred at the origin. For $\alpha > 0$ fixed, $\lambda_1 (B_m,
  \alpha)$ is a strictly decreasing, continuous function of $m > 0$.
\end{lemma}

\smallskip

\noindent{\textbf{Acknowledgements.}} The author offers his warmest
thanks Daniel Daners for many helpful discussions and suggestions.

\def\cprime{$'$}

\providecommand{\bysame}{\leavevmode\hbox to3em{\hrulefill}\thinspace}
\providecommand{\MR}{\relax\ifhmode\unskip\space\fi MR }
\providecommand{\MRhref}[2]{%
  \href{http://www.ams.org/mathscinet-getitem?mr=#1}{#2}
}
\providecommand{\href}[2]{#2}


\begin{thebibliography}{10}

\bibitem{brezis:83:af} Ha{\"{\i}}m Brezis, \emph{Analyse
    fonctionnelle}, Collection Math\'ematiques Appliqu\'ees pour la
  Ma\^\i trise, Masson, Paris, 1983, Th\'eorie et applications.

\bibitem{bucur:09:aa}
Dorin Bucur and Daniel Daners, \emph{An alternative approach to the
  {F}aber-{K}rahn inequality for {R}obin problems}, Calc. Var. Partial
  Differential Equations, to appear.

\bibitem{bucur:00:3de}
Dorin Bucur and Antoine Henrot, \emph{Minimization of the third eigenvalue of
  the {D}irichlet {L}aplacian}, R. Soc. Lond. Proc. Ser. A Math. Phys. Eng.
  Sci. \textbf{456} (2000), 985--996.

\bibitem{courant:53:mmp}
R.~Courant and D.~Hilbert, \emph{Methods of mathematical physics. {V}ol. {I}},
  Interscience Publishers, New York, N.Y., 1953.

\bibitem{dai:08:rpl}
Qiuyi Dai and Yuxia Fu, \emph{{F}aber-{K}rahn inequality for {R}obin problem
  involving $p$-{L}aplacian}, Preprint.

\bibitem{dancer:97:dpr}
E.~N. Dancer and D.~Daners, \emph{Domain perturbation for elliptic equations
  subject to {R}obin boundary conditions}, J. Differential Equations
  \textbf{138} (1997), 86--132.

\bibitem{daners:00:rbv}
Daniel Daners, \emph{Robin boundary value problems on arbitrary domains},
  Trans. Amer. Math. Soc. \textbf{352} (2000), 4207--4236.

\bibitem{daners:09:ql}
Daniel Daners and Pavel Dr{\'a}bek, \emph{A priori estimates for a class of
  quasi-linear elliptic equations}, Trans. Amer. Math. Soc., to appear.

\bibitem{edmunds:87:st}
D.~E. Edmunds and W.~D. Evans, \emph{Spectral theory and differential
  operators}, Oxford Mathematical Monographs, The Clarendon Press Oxford
  University Press, New York, 1987, Oxford Science Publications.

\bibitem{evans:92:mt}
Lawrence~C. Evans and Ronald~F. Gariepy, \emph{Measure theory and fine
  properties of functions}, Studies in Advanced Mathematics, CRC Press, Boca
  Raton, FL, 1992.

\bibitem{gilbarg:83:pde} David Gilbarg and Neil~S. Trudinger,
  \emph{Elliptic partial differential equations of second order},
  second ed., Grundlehren der Mathematischen Wissenschaften, vol. 224,
  Springer-Verlag, Berlin, 1983.

\bibitem{giorgi:05:mr}
Tiziana Giorgi and Robert~G. Smits, \emph{Monotonicity results for the
  principal eigenvalue of the generalized {R}obin problem}, Illinois J. Math.
  \textbf{49} (2005), 1133--1143 (electronic).

\bibitem{goldstein:06:gbc}
Gis{\`e}le~Ruiz Goldstein, \emph{Derivation and physical interpretation of
  general boundary conditions}, Adv. Differential Equations \textbf{11} (2006),
  457--480.

\bibitem{henrot:03:min}
Antoine Henrot, \emph{Minimization problems for eigenvalues of the
  {L}aplacian}, J. Evol. Equ. \textbf{3} (2003), 443--461, Dedicated to
  Philippe B\'enilan.

\bibitem{kennedy:08:wfk}
J.~Kennedy, \emph{A {F}aber-{K}rahn inequality for the {L}aplacian with
  generalised {W}entzell boundary conditions}, J. Evol. Equ. \textbf{8} (2008),
  557--582.

\bibitem{kennedy:09:lr2}
\bysame, \emph{An isoperimetric inequality for the second eigenvalue of the
  {L}aplacian with {R}obin boundary conditions}, Proc. Amer. Math. Soc.
  \textbf{137} (2009), 627--633.

\bibitem{ladyzhenskaya:68:lqe}
Olga~A. Ladyzhenskaya and Nina~N. Ural{\cprime}tseva, \emph{Linear and
  quasilinear elliptic equations}, Translated from the Russian by Scripta
  Technica, Inc. Translation editor: Leon Ehrenpreis, Academic Press, New York,
  1968.

\bibitem{le:06:pl}
An~L{\^e}, \emph{Eigenvalue problems for the {$p$}-{L}aplacian}, Nonlinear
  Anal. \textbf{64} (2006), 1057--1099.

\bibitem{mugnolo:06:dfw}
Delio Mugnolo and Silvia Romanelli, \emph{Dirichlet forms for general
  {W}entzell boundary conditions, analytic semigroups, and cosine operator
  functions}, Electron. J. Differential Equations (2006), No. 118, 20 pp.
  (electronic).

\bibitem{payne:57:lb}
L.~E. Payne and H.~F. Weinberger, \emph{Lower bounds for vibration frequencies
  of elastically supported membranes and plates}, J. Soc. Indust. Appl. Math.
  \textbf{5} (1957), 171--182.

\bibitem{tolksdorf:84:rqe}
Peter Tolksdorf, \emph{Regularity for a more general class of quasilinear
  elliptic equations}, J. Differential Equations \textbf{51} (1984),
  126--150.

\bibitem{vazquez:84:smp}
J.~L. V{\'a}zquez, \emph{A strong maximum principle for some quasilinear
  elliptic equations}, Appl. Math. Optim. \textbf{12} (1984), 191--202.

\bibitem{wolf:94:re}
Sven~Andreas Wolf and Joseph~B. Keller, \emph{Range of the first two
  eigenvalues of the {L}aplacian}, Proc. Roy. Soc. London Ser. A \textbf{447}
  (1994) 397--412.

\end{thebibliography}
\end{document}